\documentclass{birkjour}
\usepackage[frenchb,english]{babel}

\usepackage{mathrsfs}
\usepackage{mathtools}
\usepackage{hyperref}

\newcommand{\B}{\ensuremath{\mathrm{B}}}
\newcommand{\C}{\ensuremath{\mathbb{C}}}

\newcommand{\R}{\ensuremath{\mathbb{R}}}
\let\L\relax 
\newcommand{\L}{\mathrm{L}}
\newcommand{\w}{\mathrm{w}} 
\renewcommand{\d}{\mathop{}\mathopen{}\mathrm{d}} 
\newcommand{\rad}{\mathrm{rad}}
\newcommand{\loc}{\mathrm{loc}} 
\newcommand{\Id}{\mathrm{Id}} 
\newcommand{\SO}{\mathrm{SO}} 
\newcommand{\norm}[1]{\left\Vert#1\right\Vert}
\newcommand{\bnorm}[1]{ \big\| #1  \big\|}
\newcommand{\xra}{\xrightarrow} 
\newcommand{\co}{\colon}
\newcommand{\ot}{\otimes}

\newtheorem{thm}{Theorem}

\newtheorem{lemma}[thm]{Lemma}

\newtheorem{remark}[thm]{Remark}

\begin{document}
\title[Complementation of radial multipliers]
 {Complementation of the subspace of radial multipliers in the space of Fourier multipliers on $\R^n$}
\author[C. Arhancet]{C\'edric Arhancet}

\address{%
13 rue Didier Daurat\\
81000 Albi\\
France\\
URL: \href{http://sites.google.com/site/cedricarhancet}{http://sites.google.com/site/cedricarhancet} }

\email{cedric.arhancet@protonmail.com}

\author[C. Kriegler]{Christoph Kriegler}
\address{Laboratoire de Math\'ematiques Blaise Pascal (UMR 6620)\\
Universit\'e Clermont Auvergne\\
63 000 Clermont-Ferrand\\
France\\
URL: \href{http://math.univ-bpclermont.fr/\~kriegler/indexenglish.html}{http://math.univ-bpclermont.fr/\~{}kriegler/indexenglish.html} }

\email{christoph.kriegler@uca.fr}

\thanks{The second author was partly supported by the grant ANR-17-CE40-0021 of the French
National Research Agency ANR (project Front).}



\date{June 26, 2018}

\begin{abstract}
In this short note, we prove that the subspace of radial multipliers is contractively complemented in the space of Fourier multipliers on the Bochner space $\L^p(\R^n,X)$ where $X$ is a Banach space and where $1 \leq p <\infty$. Moreover, if $X = \C$, then this complementation preserves the positivity of multipliers.
\end{abstract}

\maketitle

%

If $|\cdot|$ denotes the euclidean norm on $\R^n$, recall that a complex function $\phi \co \R^n \to \C$ is radial if we can write $\phi(x) = \dot{\phi}(|x|)$ for some function $\dot{\phi} \co \R_+ \to \C$. If $X$ is a Banach space and if $1 \leq p<\infty$ then Theorem \ref{thm-complementation-radial-multipliers-Rn} below says that the subspace $\mathfrak{M}^{p}_\rad(\R^n,X)$ of radial Fourier multipliers on the Bochner space $\L^p(\R^n,X)$ is contractively complemented in the space $\mathfrak{M}^p(\R^n,X)$ of (scalar-valued) Fourier multipliers on $\L^p(\R^n,X)$. 

We refer to \cite[Definition 5.3.3]{HvNVW} for the definition of $\mathfrak{M}^{p}(\R^n,X)$\footnote{In the book \cite{HvNVW}, the notation $\mathfrak{M}L^p(\R^n;X)$ is used.}. We will use the notation $\mathfrak{M}^p_{\rad}(\R^n,X) 
= \left\{\phi \in \mathfrak{M}^{p}(\R^n,X)\ :\ \phi \text{ is radial}\right\}$ equipped with the norm induced by the one of the Banach space $\B(\L^p(\R^n,X))$ of bounded operators on the Bochner space $\L^p(\R^n,X)$. We say that a bounded  linear operator $T \co \L^p(\R^n) \to \L^p(\R^n)$ is positive if $T(f) \geq 0$ for any $f \geq 0$.

Let $E$ be a Hausdorff locally convex space and $E'$ its topological dual. If we denote by $E''$ the strong bidual of $E$, that is the topological dual of $E'$ where $E'$ is equipped with the strong topology, we have an injective map $j_E \co E \to E''$. This map allows us to identify $E$ as a subspace of $E''$. We say that $E$ is semireflexive if this map is surjective, see \cite[page 523]{Edw}. In this case, we can identify $E$ and $E''$ as \textit{vector spaces}. By \cite[Remarks 8.16.5]{Edw}, if $E$ is barrelled (this condition is fulfilled whenever $E$ is a Banach space) then $E'$ equipped with the weak* topology is semireflexive.

Let $\Omega$ be a locally compact space which is countable at infinity\footnote{\label{footnot2}In this case, by \cite[V.1]{Bou1}, we can remove the words ``essentially'' and ``essential'' of the  statements of \cite{Edw} used in our paper.} equipped with a Radon measure $\mu$ and let $E$ be a Hausdorff locally convex space. We say that a function $f \co \Omega \to E$ is scalarly $\mu$-integrable if for any element $\varphi$ of the topological dual $E'$ the scalar-valued function $\varphi \circ f \co \Omega \to \C$ is integrable, see \cite[page 558]{Edw}. In this case, we denote by $\int_\Omega f \d \mu \co E' \to \C$ the not necessarily continuous linear form on $E'$  defined by
\begin{equation*}
\label{Gelfand-integral}
\bigg\langle \int_\Omega f \d \mu,\varphi\bigg\rangle_{E'^*,E'}
=\int_\Omega \varphi \circ f \d \mu,	
\end{equation*}
where $E'^*$ is the algebraic dual of the topological dual $E'$. We say that a scalarly $\mu$-integrable function $f \co \Omega \to E$ is Gelfand integrable if the element $\int_\Omega f \d \mu$ belongs to $E''$ \cite[page 565]{Edw}. In this case, $\int_\Omega f \d \mu$ is called the Gelfand integral of $f$. If in addition, $\int_\Omega f \d \mu$ belongs to $E$ we have for any $\varphi \in E'$
\begin{equation}
\label{Gelfand-bracket2}
\bigg\langle \int_\Omega f \d \mu,\varphi\bigg\rangle_{E,E'}
=\int_\Omega \varphi \circ f \d \mu.
\end{equation}

By \cite[Corollary, VI.6]{Bou1} (see also \cite[Corollary 8.14.10]{Edw}), if $E$ is semireflexive and if $f \co \Omega \to E$ is a scalarly $\mu$-integrable function such that, for every compact subset $K$ of $\Omega$, $f(K)$ is bounded then $f$ is Gelfand integrable with $\int_\Omega f \d \mu \in E$.

We will use the following lemma which is a straightforward consequence of \cite[Proposition 8.14.5]{Edw}.

\begin{lemma}
\label{Lemma-1}
Let $\Omega$ be a locally compact space which is countable at infinity equipped with a Radon measure $\mu$. Let $T \co E \to F$ be a continuous linear map between Hausdorff locally convex spaces. If the function $f \co \Omega \to E$ is Gelfand integrable with $\int_{\Omega} f \d \mu\in E$ then the function $T \circ f \co \Omega \to F$ is Gelfand integrable with $\int_{\Omega} T\circ f \d \mu \in F$ and we have
\begin{equation}
\label{linear-maps-and-integrals}
T\bigg(\int_{\Omega} f \d \mu \bigg)
=\int_\Omega T \circ f \d \mu.	
\end{equation}
\end{lemma}

\begin{proof}
Since $\Omega$ is countable at infinity, by \cite[Proposition 8.14.5]{Edw}, the function $T \circ f \co \Omega \to F$ is scalarly integrable. Moreover, the same reference says that if $T'^* \co E'^* \to F'^*$ is the canonical \textit{extension} of $T\co E \to F$, we have
$$
T'^*\bigg(\int_\Omega f \d \mu\bigg)
=\int_\Omega T \circ f \d \mu.
$$
where we consider the integrals $\int_\Omega f \d \mu$ and $\int_\Omega T \circ f \d \mu$ as elements of $E'^*$ and $F'^*$. Since $\int_{\Omega} f \d \mu$ belongs to $E$, we deduce that $T'^*(\int_\Omega f \d \mu)$ belongs to $F$ and is equal to $T(\int_\Omega f \d \mu)$. We infer that $\int_\Omega T \circ f \d \mu$ also belongs to $F$ and that the equality \eqref{linear-maps-and-integrals} is true.
\end{proof}

We need the following variant of \cite[Proposition 2.5.2]{HvNVW}. 


\begin{lemma}
\label{Lemma-2}
Let $X$ be a Banach space. If a function $f$ of $\L^1_\loc(\R^n,X)$ satisfies 
$$
\int_{\R^n} \big\langle f(s), g(s) \big\rangle_{X,X'} \d s
=0
$$
for any $g \in \mathscr{D}(\R^n,X')$ then $f=0$ almost everywhere.
\end{lemma}

\begin{proof}
If $\varphi \in X'$ and if $h \in \mathscr{D}(\R^n)$, using the function $g=h \ot \varphi $ of $\mathscr{D}(\R^n,X')$, we obtain 
\begin{align*}
\int_{\R^n} h(s) \big\langle f(s), \varphi \big\rangle_{X,X'} \d s
& =\int_{\R^n} \big\langle f(s), (h \ot \varphi)(s) \big\rangle_{X,X'} \d s \\
& =\int_{\R^n} \big\langle f(s), g(s) \big\rangle_{X,X'} \d s =0.
\end{align*}
By \cite[Proposition 2.5.2]{HvNVW}, we deduce that the function $\langle f(\cdot), \varphi\rangle_{X,X'}$ is null almost everywhere. By \cite[Corollary 1.1.25]{HvNVW}, we conclude that $f=0$ almost everywhere.
\end{proof}

\begin{thm}
\label{thm-complementation-radial-multipliers-Rn}
Let $X$ be a Banach space. Suppose $1 \leq p<\infty$. Then there exists  a contractive projection $P_{p,X} \co \mathfrak{M}^{p}(\R^n,X) \to \mathfrak{M}^{p}(\R^n,X)$ onto the subspace $\mathfrak{M}^{p}_{\rad}(\R^n,X)$. Moreover, the map $M_{P_{p,\C}(\phi)} \co \L^p(\R^n) \to \L^p(\R^n)$ associated to the radial multiplier $P_{p,\C}(\phi)$ is positive if the map $M_\phi \co \L^p(\R^n) \to \L^p(\R^n)$ is positive.
\end{thm}

\begin{proof}
Let $\SO(\R^n)$ be the compact group of orthogonal mappings $R \co \R^n \to \R^n$ with determinant $1$ equipped with its normalized left Haar measure $\mu$. For $R \in \SO(\R^n)$, consider the induced map $S_R \co \L^p(\R^n) \to \L^p(\R^n)$  defined by $S_{R}(f) = f(R\, \cdot)$. Clearly, $S_R$ and its inverse $S_{R^{-1}}$ are isometric and positive. By \cite[Theorem 2.1.3]{HvNVW}, we obtain that $S_R \ot \Id_X \co \L^p(\R^n,X) \to \L^p(\R^n,X)$ is a well-defined isometric map. For any $f \in \L^p(\R^n)$, by \cite[Chap. VIII, \S2, Section 5]{Bou2}, the map $\SO(\R^n) \to \L^p(\R^n)$, $R \mapsto S_{R}(f)$ is continuous. More generally, it is easy to see that for any $f \in \L^p(\R^n,X)$ the map $\SO(\R^n) \to \L^p(\R^n,X)$, $R \mapsto (S_R \ot \Id_X)(f)$ is also continuous. Since the composition of operators is strongly continuous on bounded sets by \cite[Proposition C.19]{EFHN}, for any $\phi \in \mathfrak{M}^p(\R^n,X)$, we deduce that the map $\SO(\R^n) \to \L^p(\R^n,X)$, $R \mapsto ((S_R^{-1} M_\phi S_R)\ot \Id_X)f$ is also continuous, hence Bochner integrable on the compact $\SO(\R^n)$. Now, for any $\phi \in \mathfrak{M}^p(\R^n,X)$, put
\begin{equation}
\label{Equa33}
Q_{p,X}(\phi)f 
=\int_{\SO(\R^n)} \big((S_R^{-1} M_\phi  S_R )\ot \Id_X\big)f \d\mu(R).	
\end{equation}
For any $f \in \L^p(\R^n,X)$ and any symbol $\phi \in \mathfrak{M}^{p}(\R^n,X)$, we have
\begin{align*}
\MoveEqLeft
  \bnorm{Q_{p,X}(\phi)f}_{\L^p(\R^n,X)} 
		=\norm{\int_{\SO(\R^n)}  \big((S_R^{-1} M_\phi  S_R )\ot \Id_X\big)f\d\mu(R)}_{\L^p(\R^n,X)}\\
		&\leq \int_{\SO(\R^n)} \norm{\big((S_R^{-1} M_\phi  S_R )\ot \Id_X\big)f}_{\L^p(\R^n,X)} \d\mu(R)\\
		&\leq  \norm{M_\phi \ot \Id_X}_{\L^p(\R^n,X) \to \L^p(\R^n,X)}\norm{f}_{\L^p(\R^n,X)}.
\end{align*}
Consequently, we have a well-defined contractive map $Q_{p,X} \co \mathfrak{M}^{p}(\R^n,X) \to \B(\L^p(\R^n,X))$. 

In the sequel, we denote by $\L^1(\R^n,X)_{\w}$ the space $\L^1(\R^n,X)$ equipped with the locally convex topology $\sigma(\L^1(\R^n,X),\L^\infty(\R^n,X'))$ obtained with the inclusion $\L^\infty(\R^n,X') \subset (\L^1(\R^n,X))'$ given by \cite[Proposition 1.3.1]{HvNVW}. Note that $\L^\infty(\R^n,X')$ is norming for $\L^1(\R^n,X)$ again by \cite[Proposition 1.3.1]{HvNVW}. So we have a well-defined dual pair and furthermore $\L^1(\R^n,X)$ is Hausdorff for this topology. Similarly, we will use $\L^\infty(\R^n,X)_{\w^*}$ for the space $\L^\infty(\R^n,X)$ equipped with the topology $\sigma(\L^\infty(\R^n,X),\L^1(\R^n,X'))$, which is locally convex, obtained with the inclusion $\L^1(\R^n,X') \subset (\L^\infty(\R^n,X))'$.

Note that for any $g \in \L^1(\R^n)$, the complex function 
$$
R \mapsto \big\langle g,\phi(R^{-1} \cdot)\big\rangle_{\L^1(\R^n),\L^\infty(\R^n)}
=\big\langle S_{R}(g),\phi\big\rangle_{\L^1(\R^n),\L^\infty(\R^n)}
$$ 
is continuous on $\SO(\R^n)$ and consequently measurable. Since $\SO(\R^n)$ is compact, we deduce that this function is integrable. Consequently, the function $\SO(\R^n) \to \L^\infty(\R^n)_{\w^*}$, $R \mapsto \phi(R^{-1}\, \cdot)$ is scalarly $\mu$-integrable. Moreover, it is bounded since in $\L^\infty(\R^n)$ weak* bounded subsets coincide with bounded subsets by \cite[Theorem 2.6.7]{Meg}. Since $\L^\infty(\R^n)_{\w^*}$ is semireflexive, we deduce that this bounded function is Gelfand integrable by \cite[Corollary, VI.6]{Bou1} and that $\int_{\SO(\R^n)} \phi(R^{-1}\cdot)\d\mu(R)$ belongs to $\L^\infty(\R^n)$. 

Let $f \co \R^n \to X$ be a function of the vector-valued Schwartz space $\mathcal{S}(\R^n,X)$ such that the continuous function $\hat{f} \co \R^n \to X$ has compact support. Note that the subset of such functions is dense in the Bochner space $\L^p(\R^n,X)$ by \cite[Proposition 2.4.23]{HvNVW} and that $\hat{f}$ belongs to $\L^1(\R^n,X)$. By \cite[IV.94, Corollary 1]{Bou1}, the product $T \co \L^\infty(\R^n) \to \L^1(\R^n,X)$, $h \mapsto h\hat{f}$ is continuous. Moreover, $T$ remains continuous when considered as a map $T \co \L^\infty(\R^n)_{\w^*} \to \L^1(\R^n,X)_{\w}$. Indeed, suppose that the net $(h_i)$ of $\L^\infty(\R^n)$ converges to $h$ in the weak* topology. Then for any $g \in \L^\infty(\R^n,X')$, the function $x \mapsto \big\langle \hat{f}(x),g(x) \big\rangle_{X,X'}$ belongs to $\L^1(\R^n)$ by \cite[IV.94, Corollary 1]{Bou1}, and consequently
\begin{align*}
& \big\langle T(h_i), g \big\rangle_{\L^1(\R^n,X),\L^\infty(\R^n,X')}
=\big\langle h_i\hat{f}, g \big\rangle_{\L^1(\R^n,X),\L^\infty(\R^n,X')} \\
& =\int_{\R^n} h_i(x) \big\langle \hat{f}(x),g(x) \big\rangle_{X,X'} \d x
\xra[\ i\ ]{} \int_{\R^n} h(x) \big\langle \hat{f}(x),g(x) \big\rangle_{X,X'} \d x \\
&=\big\langle h\hat{f}, g \big\rangle_{\L^1(\R^n,X),\L^\infty(\R^n,X')}
=\big\langle T(h), g \big\rangle_{\L^1(\R^n,X),\L^\infty(\R^n,X')}.
\end{align*}  
Then Lemma \ref{Lemma-1} implies that the function $\SO(\R^n) \to \L^1(\R^n,X)_{\w}$, $R \mapsto \phi(R^{-1}\, \cdot)\hat{f}$ is Gelfand integrable with $\int_{\SO(\R^n)} \phi(R^{-1}\, \cdot)\hat{f} \d\mu(R)$ belonging to $\L^1(\R^n,X)_{\w}$ and that
\begin{equation}
\label{Second-equation}
\bigg(\int_{\SO(\R^n)} \phi(R^{-1}\, \cdot) \d\mu(R)\bigg)\hat{f}
\overset{\eqref{linear-maps-and-integrals}}=\int_{\SO(\R^n)} \phi(R^{-1}\, \cdot)\hat{f} \d\mu(R),
\end{equation}
where the integrals are Gelfand integrals and both sides belong to $\L^1(\R^n,X)$.
Now using an obvious vector-valued extension of \cite[Proposition 1.3 (vii)]{Hao} in the third and in the fifth equality, we obtain a.e.
\begin{align}
\label{first-equation}
\big((S_R^{-1} M_\phi  S_R )\ot \Id_X\big)f
&=\big((S_R^{-1} M_\phi)\ot \Id_X\big) (f(R\,\cdot)) \\
& =(S_R^{-1} \ot \Id_X) \mathcal{F}^{-1}\big[\phi(\cdot) \widehat{f(R\,\cdot)} \big]\nonumber\\
&=(S_R^{-1} \ot \Id_X) \mathcal{F}^{-1}\big[ \phi(\cdot) \hat{f} (R\,\cdot) \big] \nonumber \\
&=(S_R^{-1} \ot \Id_X)\big(\big(\mathcal{F}^{-1}\big[ \phi(R^{-1}\, \cdot) \hat{f} \big]\big)(R\, \cdot)\big) \nonumber\\
&=\mathcal{F}^{-1} \big(\phi(R^{-1}\, \cdot) \hat{f} \big). \nonumber
\end{align}
According to \cite[page 105]{HvNVW}, the inverse Fourier transform $\mathcal{F}^{-1} \co \L^1(\R^n,X) \to \L^\infty(\R^n,X)$ is bounded. We will show that it remains continuous when considered as a map $\mathcal{F}^{-1} \co \L^1(\R^n,X)_{\w} \to \L^\infty(\R^n,X)_{\w^*}$. Indeed, suppose that the net $(h_i)$ of $\L^1(\R^n,X)_{\w}$ converges to $h$. For any $g \in \L^1(\R^n,X')$, using two times an obvious vector-valued extension of \cite[Theorem 1.12]{Hao}, we obtain
\begin{align*}
\big\langle \mathcal{F}^{-1}(h_i), g \big\rangle_{\L^\infty(\R^n,X),\L^1(\R^n,X')}
&=\big\langle h_i, \mathcal{F}^{-1}(g) \big\rangle_{\L^1(\R^n,X),\L^\infty(\R^n,X')} \\
&\xra[\ i\ ]{} \big\langle h, \mathcal{F}^{-1}(g) \big\rangle_{\L^1(\R^n,X),\L^\infty(\R^n,X')} \\
&=\big\langle \mathcal{F}^{-1}(h), g \big\rangle_{\L^\infty(\R^n,X),\L^1(\R^n,X')}.
\end{align*}  
Using again Lemma \ref{Lemma-1} with $\mathcal{F}^{-1}$ instead of $T$, we obtain that the function 
$$
\SO(\R^n) \to \L^\infty(\R^n,X)_{\w^*},\ R \mapsto \mathcal{F}^{-1} \big(\phi(R^{-1}\, \cdot)\hat{f}\big)
$$ 
is Gelfand integrable with $\int_{\SO(\R^n)} \mathcal{F}^{-1} \big(\phi(R^{-1}\, \cdot)\hat{f}\big) \d\mu(R)$ belonging to the space $\L^\infty(\R^n,X)$ and that a.e.
\begin{align}
\MoveEqLeft
\label{formula}
\mathcal{F}^{-1} \bigg(\bigg(\int_{\SO(\R^n)} \phi(R^{-1}\, \cdot) \d\mu(R)\bigg) \hat{f}\bigg) \\
&\overset{\eqref{Second-equation}}=\mathcal{F}^{-1} \bigg(\int_{\SO(\R^n)} \phi(R^{-1}\, \cdot) \hat{f} \d\mu(R) \bigg)\nonumber\\
&\overset{\eqref{linear-maps-and-integrals}}=\int_{\SO(\R^n)} \mathcal{F}^{-1} \big(\phi(R^{-1}\, \cdot)\hat{f}\big) \d\mu(R) \nonumber\\
&\overset{\textrm{def}}=\int_{\SO(\R^n)}^{\mathrm{Gel}} \xi_R \d\mu(R), \nonumber
\end{align}
where the last Gelfand integral defines an element of $\L^\infty(\R,X)$. Now, we show that this Gelfand integral and the Bochner integral \eqref{Equa33} in $\L^p(\R^n,X)$ are equal almost everywhere. If $g$ belongs to $\mathscr{D}(\R^n,X')$, using \cite[(1.2) page 15]{HvNVW} in the sixth equality with the bounded map $T \co \L^p(\R^n,X) \to \C$, $h \mapsto \langle h, g \rangle_{\L^p(\R^n,X),\L^{p'}(\R^n,X')} $, we obtain
\begin{align*}
& \int_{\R^n} \bigg\langle \bigg(\int_{\SO(\R^n)}^{\mathrm{Gel}}  \xi_R \d\mu(R)\bigg)(s), g(s) \bigg\rangle_{X,X'} \d s \\
& =\bigg\langle \int_{\SO(\R^n)}^{\mathrm{Gel}}  \xi_R \d\mu(R),g \bigg\rangle_{\L^\infty(\R^n,X),\L^1(\R^n,X')} \\
& \overset{\eqref{Gelfand-bracket2}}=\int_{\SO(\R^n)} \big\langle \xi_R,g \big\rangle_{\L^\infty(\R^n,X),\L^1(\R^n,X')} \d\mu(R)  \\
& =\int_{\SO(\R^n)} \bigg(\int_{\R^n} \big\langle \xi_R(s),g(s) \big\rangle_{X,X'} \d s\bigg)\d\mu(R) \\
& \overset{\eqref{first-equation}}=\int_{\SO(\R^n)} \bigg(\int_{\R^n} \Big\langle\big(\big((S_R^{-1} M_\phi  S_R )\ot \Id_X\big)f\big)(s),g(s) \Big\rangle_{X,X'} \d s\bigg)\d\mu(R)\\
& =\int_{\SO(\R^n)} \big\langle \big((S_R^{-1} M_\phi  S_R )\ot \Id_X\big)f, g \big\rangle_{\L^p(\R^n,X),\L^{p'}(\R^n,X')} \d\mu(R)\\
& =\bigg\langle \int_{\SO(\R^n)}^{\mathrm{Boc}} \big((S_R^{-1} M_\phi  S_R )\ot \Id_X\big)f \d\mu(R), g \bigg\rangle_{\L^p(\R^n,X),\L^{p'}(\R^n,X')} \\
& =\int_{\R^n} \bigg\langle \bigg(\int_{\SO(\R^n)}^{\mathrm{Boc}}\big((S_R^{-1} M_\phi  S_R )\ot \Id_X\big)f\d\mu(R)\bigg)(s), g(s) \bigg\rangle_{X,X'} \d s\\
& \overset{\eqref{Equa33}}=\int_{\R^n} \big\langle (Q_{p,X}(\phi)f)(s), g(s) \big\rangle_{X,X'} \d s.
\end{align*}
Using Lemma \ref{Lemma-2} in the first equality, we conclude that a.e.
$$
Q_{p,X}(\phi)f
=\int_{\SO(\R^n)}^{\mathrm{Gel}} \xi_R \d\mu(R)
\overset{\eqref{formula}}=\mathcal{F}^{-1} \bigg(\bigg(\int_{\SO(\R^n)} \phi(R^{-1}\, \cdot) \d\mu(R)\bigg) \hat{f}\bigg).
$$
Using the above mentioned density, we conclude that the bounded operator
$$
Q_{p,X}(\phi) \co \L^p(\R^n,X) \to \L^p(\R^n,X)
$$
is induced by the  Fourier multiplier 
\begin{equation}
\label{symbol-psi}
P_{p,X}(\phi)
\overset{\textrm{def}}=\int_{\SO(\R^n)} \phi(R^{-1}\,\cdot) \d\mu(R)	
\end{equation} 
(Gelfand integral). Moreover, for any $R_0 \in \SO(\R^n)$, using the obvious continuity of the linear map $S_{R_0} \co \L^\infty(\R^n)_{\w^*} \to \L^\infty(\R^n)_{\w^*}$ together with Lemma \ref{Lemma-1}, we see that
\begin{align*}
S_{R_0}(P_{p,X}(\phi)) 
&\overset{\eqref{symbol-psi}}=S_{R_0}\bigg(\int_{\SO(\R^n)} \phi(R^{-1}\,\cdot) \d\mu(R) \bigg)
\overset{\eqref{linear-maps-and-integrals} }=\int_{\SO(\R^n)} S_{R_0} S_R^{-1}(\phi) \d\mu(R) \\
& =\int_{\SO(\R^n)} S_R^{-1}(\phi) \d\mu(R) =\int_{\SO(\R^n)} \phi(R^{-1}\,\cdot) \d\mu(R)
\\
&\overset{\eqref{symbol-psi}}=P_{p,X}(\phi).
\end{align*}
We deduce that the multiplier $P_{p,X}(\phi)$ is radial, i.e. $P_{p,X}(\phi)$ belongs to $\mathfrak{M}^{p}_{\rad}(\R^n,X)$. In addition, if $\phi$ itself is already radial, then 
$$
P_{p,X}(\phi)
\overset{\eqref{symbol-psi}}=\int_{\SO(\R^n)} \phi(R^{-1}\,\cdot) \d\mu(R)
=\int_{\SO(\R^n)} \phi \d\mu(R) 
=\phi.
$$ 

Finally, if $X=\C$ and if the map $M_\phi \co \L^p(\R^n) \to \L^p(\R^n)$ is positive, then $S_{R}^{-1} M_\phi S_R \co \L^p(\R^n) \to \L^p(\R^n)$ is positive, so it is easy to prove that the map $M_{P_{p,\C}(\phi)} \co \L^p(\R^n) \to \L^p(\R^n)$ is also positive by using \cite[Proposition 1.2.25]{HvNVW}.
\end{proof}


\end{document}